\documentclass[final,12pt]{elsarticle}
\usepackage{srcltx}
\usepackage{eurosym}
\usepackage{mathtools}
\usepackage{amsmath}
\usepackage{amsfonts}
\usepackage{amssymb}
\usepackage{amsthm}
\usepackage{graphicx}
\usepackage{mathrsfs}
\usepackage{xcolor}
\usepackage{exscale}
\usepackage{latexsym}
\usepackage{showkeys}

\numberwithin{equation}{section}

\usepackage[colorlinks,plainpages=true,pdfpagelabels,hypertexnames=true,colorlinks=true,pdfstartview=FitV,linkcolor=blue,citecolor=red,urlcolor=black]{hyperref}
\PassOptionsToPackage{unicode}{hyperref}
\PassOptionsToPackage{naturalnames}{hyperref}
 
\usepackage{enumerate}
\usepackage[shortlabels]{enumitem}
\usepackage{bookmark}
\usepackage{wasysym}
\usepackage{esint}
 
\usepackage[ddmmyyyy]{datetime}
 
\usepackage[margin=2cm]{geometry}
\makeatletter
\g@addto@macro\normalsize{%
	\setlength\abovedisplayskip{4pt}
	\setlength\belowdisplayskip{4pt}
	\setlength\abovedisplayshortskip{4pt}
	\setlength\belowdisplayshortskip{4pt}
}
 
\everymath{\displaystyle}
 
\usepackage[capitalize,nameinlink]{cleveref}
\crefname{section}{Section}{Sections}
\crefname{subsection}{Subsection}{Subsections}
\crefname{condition}{Condition}{Conditions}
\crefname{hypothesis}{Hypothesis}{Conditions}
\crefname{assumption}{Assumption}{Assumptions}
\crefname{lemma}{Lemma}{Lemmas}
\crefname{definition}{Definition}{Definitions}

\crefformat{equation}{\textup{#2(#1)#3}}
\crefrangeformat{equation}{\textup{#3(#1)#4--#5(#2)#6}}
\crefmultiformat{equation}{\textup{#2(#1)#3}}{ and \textup{#2(#1)#3}}
{, \textup{#2(#1)#3}}{, and \textup{#2(#1)#3}}
\crefrangemultiformat{equation}{\textup{#3(#1)#4--#5(#2)#6}}%
{ and \textup{#3(#1)#4--#5(#2)#6}}{, \textup{#3(#1)#4--#5(#2)#6}}%
{, and \textup{#3(#1)#4--#5(#2)#6}}

\Crefformat{equation}{#2Equation~\textup{(#1)}#3}
\Crefrangeformat{equation}{Equations~\textup{#3(#1)#4--#5(#2)#6}}
\Crefmultiformat{equation}{Equations~\textup{#2(#1)#3}}{ and \textup{#2(#1)#3}}
{, \textup{#2(#1)#3}}{, and \textup{#2(#1)#3}}
\Crefrangemultiformat{equation}{Equations~\textup{#3(#1)#4--#5(#2)#6}}%
{ and \textup{#3(#1)#4--#5(#2)#6}}{, \textup{#3(#1)#4--#5(#2)#6}}%
{, and \textup{#3(#1)#4--#5(#2)#6}}

\crefdefaultlabelformat{#2\textup{#1}#3}

\newtheorem{theorem} {Theorem}[section]

\newtheorem{lemma}[theorem]{Lemma}

\newtheorem{counter example}[theorem]{Counter Example}

\newtheorem{definition}[theorem] {Definition}

\def\CC{{\rm \kern.24em \vrule width.02em height1.4ex depth-.05ex \kern-.26emC}}

\def\TagOnRight

\def\AA{{it I} \hskip-3pt{\tt A}}

\def\QQ{\rlap {\raise 0.4ex \hbox{$\scriptscriptstyle |$}} {\hskip -0.1em Q}}

\makeatletter
\newcommand{\vo}{\vec{o}\@ifnextchar{^}{\,}{}}
\makeatother

\def\YYint#1#2#3{{\setbox0=\hbox{$#1{#2#3}{\iint}$}
		\vcenter{\hbox{$#2#3$}}\kern-.50\wd0}}
  
\def\XXint#1#2#3{{\setbox0=\hbox{$#1{#2#3}{\int}$}
		\vcenter{\hbox{$#2#3$}}\kern-.50\wd0}}

\makeatletter
\def\namedlabel#1#2{\begingroup
	\def\@currentlabel{#2}%
	\label{#1}\endgroup
}
\makeatother
 
\makeatletter
\newcommand{\rmh}[1]{\mathpalette{\raisem@th{#1}}}
\newcommand{\raisem@th}[3]{\hspace*{-1pt}\raisebox{#1}{$#2#3$}}
\makeatother

\newcounter{desccount}

\newcommand{\descref}[2]{\hyperref[#1]{\textcolor{black}{}\textcolor{blue}{ #2}\textcolor{black}{}}}

\newcommand{\dref}[2]{\hyperref[#1]{\textcolor{black}{(}\textcolor{blue}{\bf #2}\textcolor{black}{)}}}
 
\newcommand{\pa} {\partial}

\newcommand{\al} {\alpha}
\newcommand{\rr}{\rightarrow}
\newcommand{\ti}{\tilde}

\newcommand{\B} {\beta}

\newcommand{\e}  {\epsilon}

\newcommand{\la} {\lambda}

\newcommand{\f}{\infty}

\newcommand{\tht}{\theta}

\newcommand{\rh}{\rho}

\DeclareMathOperator{\dv}{div}

\newcommand{\norm}[1]{\left|\hspace{-0.2mm}\left| #1 \right|\hspace{-0.2mm}\right|}
\newcommand{\abs}[1]{\left| #1\right|}

\newcounter{whitney}
\refstepcounter{whitney}

\newcounter{ineqcounter}
\refstepcounter{ineqcounter}
 
\makeatletter
\def\ps@pprintTitle{%
	\let\@oddhead\@empty
	\let\@evenhead\@empty
	\def\@oddfoot{}%
	\let\@evenfoot\@oddfoot}
\makeatother
 
\usepackage[doublespacing]{setspace}
\usepackage[titletoc,toc,page]{appendix}

\makeatletter
\newcommand{\refcheckize}[1]{%
	\expandafter\let\csname @@\string#1\endcsname#1%
	\expandafter\DeclareRobustCommand\csname relax\string#1\endcsname[1]{%
		\csname @@\string#1\endcsname{##1}\wrtusdrf{##1}}%
	\expandafter\let\expandafter#1\csname relax\string#1\endcsname
}
\makeatother
 
\refcheckize{\cref}
\refcheckize{\Cref}

 

\makeatletter
\newcommand{\mainsectionstyle}{%
	\renewcommand{\@secnumfont}{\bfseries}
	\renewcommand\section{\@startsection{section}{2}%
		\z@{.5\linespacing\@plus.7\linespacing}{-.5em}%
		{\normalfont\bfseries}}%
}
\makeatother
 
\usepackage{pgf,tikz}
\usetikzlibrary{arrows}
\usetikzlibrary{decorations.pathreplacing}

\usepackage{xpatch}
\xpatchcmd{\MaketitleBox}{\hrule}{}{}{} 
\xpatchcmd{\MaketitleBox}{\hrule}{}{}{}

\date{}
\newcommand{\R}{\mathbb{R}}

\newcommand{\G}{\mathbb{G}}
\newcommand{\Pp}{\mathbb{P}}

\newcommand{\F}{\mathfrak{F}}

\newcommand{\E}{\mathbb{E}}

\newcommand{\U}{\mathcal{U}}

\newcommand{\rd}{\mathrm{d}}
 
\newcommand{\dom}{\mathbb{T}^n}
 
\newcommand{\bvr}{\pmb{\varphi}}

\newcommand{\mbf}{\textbf{m}}

	\date{}
	 
	\newcommand{\inpr}[3][]{\left\langle#2 \,,\, #3\right\rangle_{#1}}

	\allowdisplaybreaks
	\everymath{\displaystyle}
	 
	\newcommand{\vbf}{\mathrm{\textbf{v}}}
	
\begin{document}

	\begin{frontmatter}
		\title{Energy-balance for    the incompressible Euler
 equations with stochastic forcing} 
		
		\author[myaddress1]{Shyam Sundar Ghoshal}\ead{ghoshal@tifrbng.res.in}
		\author[myaddress1]{Animesh Jana}\ead{animesh@tifrbng.res.in}
		 
		\author[myaddress1]{Barun Sarkar}\ead{barunsarkar.math@gmail.com}
		\address[myaddress1]	{Tata Institute of Fundamental Research,Centre For Applicable Mathematics,
			Sharada Nagar, Chikkabommsandra, Bangalore 560065, India.}
		
		\begin{abstract}
		  We establish energy-balance for weak solutions of  the stochastically forced incompressible Euler equations, enjoying H\"older regularity $C^{\al}$, $\al>1/3$.  It is well known as the Onsager's conjecture for the   deterministic incompressible Euler equations, which describes the energy conservation of weak solutions having H\"older regularity $C^{\al}$, $\al>1/3$. Additionally, we obtain  energy-balance for the  inhomogeneous incompressible Euler system driven by cylindrical Wiener process. 
		\end{abstract}
		
 
	\end{frontmatter}
	
	\tableofcontents
	
	\section{Introduction}\label{intro}
	
	In this paper we prove  Energy-balance of the following two incompressible Euler equations driven by cylindrical
  Wiener process -
	
	\begin{enumerate}[label={\bf(Eq.:\arabic*)}]
	\item\label{eq1} Incompressible Euler
 equations with stochastic forcing:-
	\begin{align}
  & \rd \vbf + \left[ \dv_x(\vbf\otimes\vbf) + \nabla_x p\right]\rd t =   \G(\vbf)\, \rd W_t,\label{sto-v}\\
  & \dv_x\vbf =0.\label{dv-v}
  \end{align}
	\item\label{eq2} Inhomogeneous incompressible Euler
 equations with stochastic forcing:-
	\begin{align} 
  & \rd\rho + \dv_x(\rh\vbf)\, \rd t = 0,\label{cont-Eq}\\
  & \rd(\rh\vbf) + \dv_x\left(  \rh\vbf\otimes\vbf \right)\rd t + \nabla_x p\, \rd t = \ti\G(\rho,\rh\vbf) \rd W_t,\label{moment-Eq}\\
  & \dv_x   \vbf   = 0\label{dv-vel}. 
 \end{align}
	\end{enumerate}
	
	In above equations vector $\vbf(t,x)$ denotes the velocity of the fluid particle which occupies the point $x$ at time $t$,  $p(t,x)$ denotes the hydrodynamic pressure and $\rho(t,x)>0$ is the scalar density of a fluid. Both the equations are driven by a cylindrical Wiener process $\{W_t\}_{t\geq0}$ in a separable Hilbert space $\U$, defined on some filtered probability space $(\Omega,\F, \{\F_t\}_{t\geq0},\Pp)$ with a complete, right-continuous filtration. We work with periodic boundary conditions with period box $\dom\equiv[0,1]^n$ for $n=2,3$. The diffusion coefficients $\G$ and $\ti\G$ take values in $L_2(\U;L^2(\dom))$ which is the space of Hilbert-Schmidt operators.

	In case of deterministic incompressible  Euler equation, the famous Onsager's conjecture \cite{onsager} describes energy conservation \cite{cwt} for weak solutions which are  H\"older regular $C^{\alpha}$, for $\alpha>1/3$. Another direction of Onsager's conjecture i.e. when weak solutions are $C^\al, \alpha<1/3$, then energy may not be conserved, was shown in \cite{blsv,is}. Let us mention some other existing literature about Onsager's conjecture in the   deterministic setup - e.g.  \cite{bt,dgSG,dlszjr,fjwi}. 
	 Onsager's conjecture for general system of conservation laws are presented in  \cite{bgSGtw,bgSGtw-xtn}. 
	 Energy conservation for deterministic inhomogeneous incompressible Euler equations has been proved in \cite{fgsgw}. Existence and uniqueness results in various situations for  stochastically forced incompressible Euler equations have been studied in \cite{bes-fln,br-moy,brez-pes,cap-cut,glho-vil}.

	The goal of this article is to  investigate stochastic counterpart of energy-balance of the incompressible Euler equation, whose weak solutions are $C^{\al}$ H\"older regular, for $\al>1/3$.
	 Last part of the paper is devoted to energy-balance for the inhomogeneous incompressible Euler equation driven by cylindrical Wiener process. We prove in more general setting by considering solutions in Besov space. 
	 Note that, due to the presence of a noise term in the systems, we obtain energy balance equations corresponding to both the systems \ref{eq1} and \ref{eq2}. 
 


	The paper is organized as follows. We describe Besov space, cylindrical Wiener process, diffusion coefficients and required lemmas in section \ref{set-def}. Energy-balance for the stochastically forced incompressible Euler equation is presented in section  \ref{incompressible-Euler} and at the end in section \ref{INhomognELR} we prove energy balance for the inhomogeneous incompressible Euler equation driven by cylindrical Wiener process.  
	
	\section{Preliminaries}\label{set-def}
  
  In this present article, we consider that weak  solutions possess Besov regularity, defined as follows: let $\alpha\in(0,1)$, $1\leq q<\infty$. Then for an open bounded set $ \Pi\subset\bar{ \Pi}\subset\dom$, Besov space $B_q^{\alpha,\infty}(\Pi,\R^n)$ is a subset of $L^q(\Pi,\R^n)$ such that $f\in {B_q^{\alpha,\infty}(\Pi,\R^n)}$ if the following seminorm is finite, 
  \[|f|_{B_q^{\alpha,\infty}(\Pi,\R^n)} := \sup_{0\neq\zeta\in\R^n,\zeta+\Pi\subset\dom} \frac{\| f(\cdot +\zeta) - f(\cdot)\|_{L^q(\Pi,\R^n)}}{|\zeta|^{\alpha} }<\f.\]
  We note that $C^\al(\Pi)\subset B^{\al,\f}_q(\Pi)$ for $1\leq q<\f$. 
  
  The cylindrical Wiener process $\{W_t\}_{t\geq0}$ considered here are defined over a separable Hilbert space $\U$ with respect to a filtered probability space $(\Omega,\F, \{\F_t\}_{t\geq0},\Pp)$ where $\{\F_t\}_{t\geq0}$ is a complete and right-continuous filtration. $W_t$ can be the formally written as
 \begin{equation}\label{defn-W}
 W_t:=\sum_{j=1}^{\infty}e_jW^j_t,
 \end{equation}
 where, $\{e_j\}_{j\geq1}$ is an orthonormal basis of $\U$ and $\{W^j_t\}_{j\geq1}$ is a sequence of mutually independent $\R$-valued Brownian motions with respect to $(\Omega,\F, \{\F_t\}_{t\geq0},\Pp)$. We describe the diffusion coefficients $\G$ and $\ti\G$ respectively. 
 
 Let $\vbf\in L^2(\dom)$, then we define the noise coefficient $\G(\vbf):\U\to L^2(\dom;\R^n)$ by
 \[\mathbb{G}(\vbf)e_j:=\G_j(\cdot, \vbf(\cdot)).\]
 The coefficients $\G_j=\G_j(x,\vbf):\dom \times \R^n\to\R^n$ are $C^{1}$- functions  that satisfy uniformly in $x\in\dom$   
 \begin{align}\label{con-G}
 \begin{split}
 & \G_j(\cdot,0)=0,\\
 & |\nabla_{\vbf}\G_j| \leq g_j,\ \ \text{[where sequence $\{g_j\}\subset(0,\f)$ ]}\\
 & \sum_{j\geq1}  g_j^2<\infty.
 \end{split}
 \end{align}
 If $\G$ satisfies \eqref{con-G} and $\vbf$ are $\{\F_t\}$- progressively measurable $L^2(\dom)$- valued process such that
 \[ \vbf \in L^2\left(\Omega\times[0,T]; L^2(\dom;\R^n)\right),
 \]
 then the following integral is well-defined
 \[\int_0^t \G(\vbf)\ dW_s = \sum_{j\geq1}\int_0^t \G_j(\cdot,\vbf)\ dW^j_s.\]
 
 Let $\rho\in\ L^2(\dom)$, $\rho\geq0$ and $\mbf:=(\rh\vbf)\in L^2(\dom)$, then we define the diffusion coefficient $\ti\G(\rho,\mbf):\U\to L^2(\dom;\R^n)$ by
\[\ti\G(\rho,\mbf)e_j:=\ti\G_j(\cdot,\rho(\cdot),\mbf(\cdot)),\]
where coefficients $\ti\G_j=\ti\G_j(x,\rho,\mbf):\dom\times [0,\infty)\times \R^n\to\R^n$ are $C^{1}$- functions and there exists a non-negative real numbers sequence $\{\ti g_j\}$ such that the following holds uniformly in $x\in\dom$,  
\begin{align}\label{con-G-hat}
\begin{split}
 \ti\G_j(\cdot,0,0)=0,\ \ \ |\partial_{\rho}\ti\G_j| + |\nabla_{\mbf}\ti\G_j| \leq \ti g_j\ \ \ \mbox{ and }\ \ \ \sum_{j\geq1} {\ti g_j}^2<\infty.
\end{split}
\end{align}

 Note that, if $\ti\G$ satisfies \eqref{con-G-hat} and $\rho,(\rh\vbf)$ are $\{\F_t\}$- progressively measurable $L^2(\dom)$- valued process such that
 \begin{align*}
 \rho \in L^2\left(\Omega\times[0,T]; L^2(\dom)\right)\mbox{ and } (\rh\vbf) \in L^2\left(\Omega\times[0,T]; L^2(\dom;\R^n)\right),
 \end{align*}
 then the following is a well-defined $\{\F_t\}$- martingale in $L^2(\dom;\R^n)$, 
 \[\int_0^t \ti\G(\rho,\rh\vbf)\rd W_s := \sum_{j\geq1}\int_0^t \ti\G_j(\cdot,\rho,\rh\vbf)\rd W^j_s.\]
 
 The infinite sum in \eqref{defn-W} does not converge in   probabilistic sense as a random variable in $\U$. But, we can construct an auxiliary space $\U_0\supset\U$, where the sum converges. Define the space
 \[\U_0:= \left\{u = \sum_{j\geq1}u_je_j;\ \sum_{j\geq1}\frac{u_j^2}{j^2}<\infty\right\},\]
 with the following norm
 \[\|u\|_{\U_0}^2 := \sum_{j\geq1}\frac{u_j^2}{j^2}.\]
 Note that the embedding $\U\hookrightarrow\U_0$ is Hilbert-Schmidt. The trajectories of $\{W_t\}\subset C([0,T];\U_0)$, $\Pp$- a.s. In this article, the separable Hilbert space is chosen to be $\U=L^2(\dom)$.
 
 For detailed discussions about cylindrical Wiener process one can look \cite[chapter 2,3]{bfh-book}, \cite[chapter 2]{gm}, \cite{dpZ}, \cite{proc}.

  Through out the article we assume $\nu^\e:=\nu*\eta_\e$, where $\{\eta_\e\}_{\e>0}$ denotes the standard mollifiers sequence in space variable.
  
 We state the following three lemmas, which we use to prove Energy-balance theorems. 
  
\begin{lemma}[\cite{gjs}]\label{lemma:limit_pass_diff}
		Let $q\geq2$. Let $U_t$ be a process on $(\Omega,\mathfrak{F},(\mathfrak{F}_t)_{t\geq0},\Pp)$ such that
	\[\label{integrability:Y}
	U_t\in L^{\f}(0,T;L^{q}(\dom)),\mbox{ with }\E\left[\sup\limits_{t\in[0,T]}\norm{U_t}^{2\la}_{L^{q}}\right]<\f\mbox{ for }1\leq\lambda<\f.
	\] 
	Let $V_t$ be progressively measurable with $V_t\in L^{2}(\Omega;L^2(0,T;L_2(\U;L^{\frac{q}{q-1}}(\dom)))$ with
	\[\label{integrability:D}
\E\left(\sum\limits_{k\geq1}\int\limits_{0}^{T}\norm{V_t(e_k)}_{L^{\frac{q}{q-1}}(\dom)}^2\,\rd t\right)^{\la}<\f\mbox{ for }1\leq \la<\f.
	\] 
    Then the following holds, up to a subsequence, $\Pp$-a.s. as $\e\to0$
	\[ 
	\int\limits_{0}^{T}\left(\int\limits_{\dom} U_t^\e V_t^\e\, \rd x\right)\rd W_t\rr\int\limits_{0}^{T}\left(\int\limits_{\dom} U_t V_t\, \rd x\right)\rd W_t. 
	\] 
\end{lemma}
 The lemma \ref{lemma:limit_pass_diff} can be proved by applying B\"urkhholder-Davis-Gundy inequality \cite{bfh-book,gm}. Here we omit the proof and it can be found in the appendix of \cite{gjs}. 

\begin{lemma}\label{lemma:limit_pass_drift1}
	Let $Z_t,Y_t$ be processes on $(\Omega,\mathfrak{F},(\mathfrak{F}_t)_{t\geq0},\Pp)$ such that
	\begin{align*}
	& Z_t\in L^{\f}(0,T;L^{p}(\dom,\R^N)),\mbox{ with }\E\left[\sup\limits_{t\in[0,T]}\norm{Z_t}^{2\la}_{L^{p}}\right]<\f\mbox{ for }1\leq\lambda<\f,\\
	& Y_t\in L^{\f}(0,T;L^{q}(\dom,\R^N)),\mbox{ with }\E\left[\sup\limits_{t\in[0,T]}\norm{Y_t}^{2\la}_{L^{q}}\right]<\f\mbox{ for }1\leq\lambda<\f,
	\end{align*}
	such that $1\leq p,q\leq \f$ satisfying $p^{-1}+q^{-1}=1$. 
  Then, for a subsequence, we get the following $\Pp$-a.s. as $\e\to0$
	\[ 
	\int\limits_{0}^{T}\int\limits_{\dom}Z_t^{\e}\cdot Y_t\,\rd x\rd t\rr	\int\limits_{0}^{T}\int\limits_{\dom}Z_t\cdot Y_t\,\rd x\rd t.
	\] 
\end{lemma}
  The proof of Lemma \ref{lemma:limit_pass_drift1} is standard and it follows from properties of $L^p$ functions. 

The following commutator estimate lemma is useful for proving energy balance equation.
      \begin{lemma}[Commutator estimate, \cite{cwt,fgsgw, gmsg}]\label{lemma:commutator}
		Let $K,M\in\mathbb{N}$ and $\mathscr{O},\mathscr{O}_1\subset\R^d$ be two open sets which are bounded and $\bar{\mathscr{O}}\subset\mathscr{O}_1$. Let $f\in B^{\al,\f}_3(\mathscr{O}_1,\R^M)$ and $\phi\in B^{\B,\f}_3(\mathscr{O}_1,\R^{M})$ for $\al,\B\in(0,1)$. Let $\mathcal{H}:\mathscr{U}\rr\R^{K\times M}$ be a $C^2$ function where $\mathscr{U}$ is a bounded open convex set such that it contains the closure of range of $f$. Then we get
		\[ 
		\|\left(\mathcal{H}(f^\e)-\mathcal{H}(f)^\e\right):\nabla_x \phi^\e\|_{L^{1}(\mathscr{O})}\leq C_0\abs{f}_{B^{\al,\f}_3(\mathscr{O}_1)}^2\abs{\phi}_{B^{\B,\f}_3(\mathscr{O})}\e^{2\al+\B-1},
		\] 
		where $C_0$ depends on $\sup\{\abs{\nabla_u^2\mathcal{H}};\,u\in\mathscr{U}\}$ and domain $\mathscr{O}$.
 \end{lemma}
 The proof of Lemma \ref{lemma:commutator} is omitted here. We refer \cite{fgsgw, gmsg} for a complete proof of it.

 \section{Energy-balance for the stochastically forced  incompressible Euler
 equation}\label{incompressible-Euler}
 
 This section deals with energy-balance for the incompressible Euler equation driven by cylindrical Wiener process.
 
 \begin{definition}[pathwise weak solutions of incompressible Euler
 equations with stochastic forcing]\label{def-wk}
 Let $(\Omega,\mathfrak{F},\{\mathfrak{F}_t\}_{0\leq t\leq T},\mathbb{P})$ be a complete filtered probability space with right-continuous filtration, $\{W_t\}_{t\geq0}$ be a cylindrical Wiener process with respect to filtration $\{\F_t\}_{0\leq t\leq T}$ and $\G$ satisfy condition \eqref{con-G}. We say that $(\vbf,p,\tau)$ is a pathwise weak solution of \eqref{sto-v}-\eqref{dv-v}, if the following holds:
 \begin{enumerate}[(i)]
  \item $\tau>0$ is an a.s. strictly positive $\{\mathfrak{F}_t\}$-stopping time.
  
  \item For each $\pmb{\varphi}\in C_c^{\f}(\dom,\R^n)$,  $t\mapsto\int\limits_{\dom}\vbf\cdot\pmb{\varphi}\,dx\in C([0,\tau))$ and the process $t\mapsto \int\limits_{\dom}\vbf\cdot\bvr\,dx$, is progressively measurable with respect to filtration $(\F_t)$, such that $\Pp$-a.s.
		\[
		\vbf\in C\left([0,\tau);L^{2}(\dom,\R^n)\right)\cap L^{\f}([0,\tau)\times\dom;\R^n).
		\]
		
		\item For each $\varphi\in C_c^\f(\dom)$, $t\mapsto\int\limits_{\dom}\phi p(\cdot,t)\,dx$  is progressively measurable with respect to filtration $(\F_t)$, such that $\Pp$-a.s. $p\in L^2\left([0,\tau)\times\dom,\R\right).$
		
 \item The following equation holds $\Pp$-a.s. $\forall\,\phi\in C_c^{\f}(\dom)$ for $0\leq t<\tau$
 \[\int_{\dom}\vbf(t,\cdot)\cdot \nabla_x\phi\, \rd x=0.\]

 \item For all $\bvr\in C_c^{\f}(\dom,\R^n)$ the following equation holds $\Pp$-a.s., for $0\leq t<\tau$
 \begin{align*}
  \int\limits_{\dom}\vbf(t,\cdot)\cdot\bvr \, \rd x =& \int\limits_{\dom}\vbf(0,\cdot)\cdot\bvr   \, \rd x\\
  & + \int_0^t\int_{\dom}\left[ (\vbf\otimes\vbf) : \nabla_x \bvr +  p(x) \dv_x\bvr\right]\rd x\rd s + \int_0^t\left( \int_{\dom}\G(\vbf)\cdot\bvr\, \rd x\right) \rd W_s.
 \end{align*} 
		
 \end{enumerate}

  \end{definition}

 \begin{theorem}[Energy-balance for the stochastically forced   incompressible Euler
 equations]\label{energy}
 Let $(\vbf,p,\tau)$ be a pathwise weak solution of \eqref{sto-v}-\eqref{dv-v}, as in definition \ref{def-wk} and 
 \[\vbf\in L^3\left([0,\tau);B_3^{\alpha,\infty}(\dom,\R^n)\right).\]
 Then, for $\alpha>\frac{1}{3}$, the following Energy-balance equation holds, $\Pp$- a.s., for each $t\in[0,\tau)$
 \begin{align}\label{enrg-bln-eqn} 
   \begin{split}
   \int_{\dom}|\vbf (t,x)|^2\rd x &= \int_{\dom}|\vbf (0,x)|^2\rd x + \sum_{k\geq1}\int_0^t \int_{\dom} |\G(\vbf(s,x)) (e_k)|^2\, \rd x\, \rd s \\
   & \quad + \int_0^t \int_{\dom} 2\  \vbf(s,x)\cdot \G(\vbf(s,x))\,  \rd x\, \rd W_s.
   \end{split}
   \end{align}

  \end{theorem}
%
%
%
%
   \begin{proof}[Proof of Theorem \ref{energy}.]

 After mollifying \eqref{sto-v}-\eqref{dv-v} in space, we obtain
 \begin{align}
  & \rd \vbf^{\epsilon} + \left[ \dv_x(\vbf^{\epsilon}\otimes\vbf^{\epsilon}) + \nabla_x p^{\epsilon}\right]\rd t = \dv_x \left[ (\vbf^{\epsilon}\otimes\vbf^{\epsilon}) - (\vbf\otimes\vbf)^{\epsilon}\right]\rd t +   \G(\vbf)^{\epsilon}\rd W_t,\label{sto-vM}\\
  & \dv_x\vbf^{\epsilon} =0.\label{dv-vM}
  \end{align} 
 Strong solution of \eqref{sto-vM}, for $0\leq t<\tau$, $\Pp$-a.s.
 \begin{align*}
 \vbf^{\epsilon}(t,\cdot) =& \vbf^{\epsilon}(0,\cdot) + \int_0^t \left( - \left[ \dv_x(\vbf^{\epsilon}\otimes\vbf^{\epsilon}) + \nabla_x p^{\epsilon}\right] + \dv_x \left[ (\vbf^{\epsilon}\otimes\vbf^{\epsilon}) - (\vbf\otimes\vbf)^{\epsilon}\right] \right)\rd s \\
 & \quad + \int_0^t \G(\vbf)^{\epsilon}\rd W_s.
 \end{align*}
  Writing It\^o's formula for $r\mapsto\|r\|_{L^2(\mathbb{T}^n)}^2$, \cite[Theorem 2.10]{gm}, for $0\leq t<\tau$, $\Pp$ a.s., we obtain
 \begin{align*} 
 \|\vbf^{\epsilon}(t,x)\|_{L^2(\dom)}^2 &= \|\vbf^{\epsilon}(0,x)\|_{L^2(\dom)}^2 - \int_0^t \inpr[L^2(\dom)]{2\ \vbf^{\epsilon}(s,x)}{\left[ \dv_x(\vbf^{\epsilon}\otimes\vbf^{\epsilon}) + \nabla_x p^{\epsilon}\right](s,x)} \rd s\\
 & \quad + \int_0^t \inpr[L^2(\dom)]{2\   \vbf^{\epsilon}(s,x)}{\dv_x \left[ (\vbf^{\epsilon}\otimes\vbf^{\epsilon}) - (\vbf\otimes\vbf)^{\epsilon}\right](s,x)}  \rd s\\
 & \quad + \sum_{k\geq1}\int_0^t \|\G(\vbf(s,x))^{\epsilon}(e_k)\|_{L^2(\dom)}^2\rd s + \int_0^t \inpr[L^2(\dom)]{2\  \vbf^{\epsilon}(s,x)}{\G(\vbf(s,x))^{\epsilon}\rd W_s}.
 \end{align*}
 Now applying stochastic Fubini \cite[Theorem 2.8]{gm}, we get
 \begin{align}\label{ito-dom}
 \begin{split}
 \int_{\dom}|\vbf^{\epsilon}(t,x)|^2\rd x &= \int_{\dom}|\vbf^{\epsilon}(0,x)|^2 \rd x - \int_0^t\int_{\dom} 2\ \vbf^{\epsilon}(s,x)\cdot \left[ \dv_x(\vbf^{\epsilon}\otimes\vbf^{\epsilon}) + \nabla_x p^{\epsilon}\right](s,x)\, \rd x\,  \rd s\\
 & \quad + \int_0^t \int_{\dom} 2\   \vbf^{\epsilon}(s,x)\cdot \dv_x \left[ (\vbf^{\epsilon}\otimes\vbf^{\epsilon}) - (\vbf\otimes\vbf)^{\epsilon}\right](s,x)\, \rd x\,  \rd s\\
 & \quad + \sum_{k\geq1}\int_0^t \int_{\dom} |\G(\vbf(s,x))^{\epsilon}(e_k)|^2\, \rd x\, \rd s + \int_0^t \int_{\dom} 2\  \vbf^{\epsilon}(s,x)\cdot \G(\vbf(s,x))^{\epsilon}\, \rd x\, \rd W_s.
 \end{split}
 \end{align}
 As $\epsilon\to0$, by lemma \ref{lemma:limit_pass_drift1} we have $\Pp$-a.s.,
 \begin{align}\label{v-conv}
  \begin{split}
  & \int_{\dom}|\vbf^{\epsilon}(t,x)|^2\rd x \to \int_{\dom}|\vbf (t,x)|^2\rd x,\\
  & \int_{\dom}|\vbf^{\epsilon}(0,x)|^2\rd x \to \int_{\dom}|\vbf (0,x)|^2\rd x.
  \end{split}
 \end{align}
 Then, consider the integrand of the 2nd term of r.h.s. of \eqref{ito-dom} and apply integration by parts together with \eqref{dv-vM}, to obtain
 \begin{align}\label{2nd-trm}
  \begin{split}
  & \int_{\dom} 2\ \vbf^{\epsilon}(s,x)\cdot \left[ \dv_x(\vbf^{\epsilon}\otimes\vbf^{\epsilon}) + \nabla_x p^{\epsilon}\right](s,x)\, \rd x\\
  & = - \int_{\dom} 2\ \nabla_x\vbf^{\epsilon}(s,x) :  (\vbf^{\epsilon}\otimes\vbf^{\epsilon}) (s,x)\, \rd x - \int_{\dom} 2\ \dv_x\vbf^{\epsilon}(s,x)\,  p^{\epsilon}(s,x)\, \rd x\\  
  & = - \int_{\dom} 2\ \vbf^{\epsilon}(s,x) \cdot\nabla_x \frac{1}{2} |\vbf^{\epsilon}(s,x)|^2\ \rd x - 0 \\
  & = \int_{\dom} \dv_x \vbf^{\epsilon}(s,x)\  |\vbf^{\epsilon}(s,x)|^2\ \rd x\\
  & = 0.
  \end{split}
 \end{align}
 By Commutator estimate lemma \ref{lemma:commutator}, as $\epsilon\to0$, 3rd term of r.h.s. of \eqref{ito-dom} 
 \begin{equation}\label{3rd-trm}
 \int_0^t \int_{\dom} 2\   \vbf^{\epsilon}(s,x)\cdot \dv_x \left[ (\vbf^{\epsilon}\otimes\vbf^{\epsilon}) - (\vbf\otimes\vbf)^{\epsilon}\right](s,x)\, \rd x\,  \rd s \to 0,\,\Pp-\mbox{a.s.}
 \end{equation}
 Again, by lemma \ref{lemma:limit_pass_drift1}, as $\epsilon\to0$, 4th term of the r.h.s. of \eqref{ito-dom}
 \begin{equation}\label{4th-trm}
 \sum_{k\geq1}\int_0^t \int_{\dom} |\G(\vbf(s,x))^{\epsilon}(e_k)|^2\, \rd x\, \rd s \to \sum_{k\geq1}\int_0^t \int_{\dom} |\G(\vbf(s,x)) (e_k)|^2\, \rd x\, \rd s,\,\Pp-\mbox{a.s.}
 \end{equation} 
 At the end, as $\epsilon\to0$, by lemma \ref{lemma:limit_pass_diff}, 5th term of r.h.s. of \eqref{ito-dom}
 \begin{equation}\label{5th-trm}
 \int_0^t \int_{\dom} 2\  \vbf^{\epsilon}(s,x)\cdot \G(\vbf(s,x))^{\epsilon}\, \rd x\, \rd W_s \to \int_0^t \int_{\dom} 2\  \vbf(s,x)\cdot \G(\vbf(s,x))\,  \rd x\, \rd W_s,
 \end{equation}
 converges in mean square, hence in probability. Therefore, combining \eqref{v-conv}-\eqref{5th-trm} in \eqref{ito-dom}, as $\e\to0$, for a subsequence, we obtain  energy-balance equation \eqref{enrg-bln-eqn}, in $\Pp$- a.s. sense, for each $t\in[0,\tau)$. This concludes the proof.
 
 \end{proof}

 \section{Energy-balance for the stochastically forced inhomogeneous incompressible Euler
 equation}\label{INhomognELR}
 
 In this section, we prove energy-balance equation for the inhomogeneous incompressible Euler equation driven by the cylindrical Wiener process. 

 \begin{definition}[pathwise weak solutions of inhomogeneous incompressible Euler
	equations with stochastic forcing]\label{inhom-defn}
	
	Let $(\Omega,\mathfrak{F},\{\mathfrak{F}_t\}_{0\leq t\leq T},\mathbb{P})$ be a complete filtered probability space with right-continuous filtration, $\{W_t\}_{t\geq0}$ be a cylindrical Wiener  process with respect to the filtration $\{\F_t\}_{0\leq t\leq T}$ and $\ti\G$ satisfy condition \eqref{con-G-hat}.  $(\rho,\vbf,p,\tau)$ is called a pathwise weak solution of \eqref{cont-Eq}-\eqref{dv-vel}, if the following holds:
	
	\begin{enumerate}[(i)]
		\item $\tau>0$ is an a.s. strictly positive $\{\mathfrak{F}_t\}$-stopping time.
		\item The density $\rho$ is progressively measurable with respect to the filtration $\{\F_t\}_{0\leq t\leq T}$. There exists $\bar{r}>0$ such that the following holds for $\Pp$-a.s.
		\[\rho\geq\bar r\mbox{ and } \ \rho\in C\left([0,\tau);L^2(\dom)\right)\cap L^{\f}\left([0,\tau)\times\dom\right).\]
		\item For each $ \Phi\in C_c^{\f}(\dom,\R^n)$,  $t\mapsto\int\limits_{\dom}(\rh\vbf)\cdot \Phi\, \rd x\in C([0,\tau))$ and the process $t\mapsto \int\limits_{\dom}(\rh\vbf)\cdot\Phi\,\rd x$, is progressively measurable with respect to the filtration $\{\F_t\}_{0\leq t\leq T}$, such that for $0\leq t<\tau$, $\Pp$-a.s.
		\[ (\rh\vbf)\in C\left([0,\tau);L^2(\dom;\R^n)\right)\cap L^{\f}\left([0,\tau)\times\dom;\R^n\right).\]
		\item The pressure $p$ is progressively measurable with respect to the filtration $\{\F_t\}_{0\leq t\leq T}$ and following holds $\Pp$-a.s.
		\[ p\in C\left([0,\tau);L^2(\dom)\right)\cap L^{\f}\left([0,\tau)\times\dom\right).\]
		\item For all $\phi\in C_c^{\f}(\dom)$, the map $t\mapsto\int_{\dom}\rho\, \phi\, \rd x$, is progressively measurable with respect to the filtration $\{\F_t\}_{0\leq t\leq T}$, such that the following equation holds $\Pp$-a.s., for $0\leq t<\tau$,
		\[\int_{\dom}\rho(t,\cdot)\, \phi\, \rd x = \int_{\dom}\rho(0,\cdot)\, \phi\, \rd x\, + \int_0^t \int_{\dom} \rh\vbf(s,\cdot)\cdot\nabla_x\phi\, \rd x\, \rd s.\]
		\item The following equation holds $\Pp$-a.s., $\forall\, \Phi\in C_c^{\f}(\dom,\R^n)$ for $0\leq t<\tau$,
		\begin{align*}
		& \int_{\dom}\rh\vbf(t,\cdot)\cdot\Phi\, \rd x = \int_{\dom} \rh\vbf(0,\cdot)\cdot\Phi\, \rd x\\
		&\quad\quad + \int_0^t\int_{\dom} \left[ \left(  \rh\vbf\otimes\vbf \right) : \nabla_x\Phi\, + p\, \dv_x\Phi\right] \rd x\, \rd s + \int_0^t \left( \int_{\dom}\ti\G(\rho,\rh\vbf)\cdot\Phi\, \rd x\right)\rd W_s.
		\end{align*}
		\item The following equation holds $\Pp$-a.s., $\forall \,\psi\in C_c^{\f}(\dom)$ for $0\leq t<\tau$,
		\[\int_{\dom} \vbf(t,\cdot)\cdot\nabla_x\psi\, \rd x = 0.\]
	\end{enumerate}
	
\end{definition}	

\begin{theorem}[Energy-balance for the stochastically forced inhomogeneous incompressible Euler
 equation]\label{inhom-energyBLN-theorem}
 Let $(\rho,\vbf,p,\tau)$ be a pathwise weak solution of \eqref{cont-Eq}-\eqref{dv-vel}, as in definition \ref{inhom-defn} and 
 \[(\rho,\vbf,p) \in L^3\left([0,\tau);B_3^{\alpha,\infty}(\dom,(0,\f)\times\R^n\times\R)\right).\]
 Then, for $\alpha>\frac{1}{3}$, the following Energy-balance equation holds, $\Pp$- a.s., for all $\tht\in C^{\f}_c(0,\f)$ and $\Psi\in C^{\f}_c(\dom)$, for each $t\in[0,\tau)$
 \begin{align}\label{enr-inhom-bln-EQUATION}
  \begin{split}
 & \int_0^t \pa_t\tht \int_{\dom} \left( \frac{1}{2}\rh|\vbf|^2 \right)\Psi\, \rd x\rd s + \int_0^t \tht\int_{\dom}  \left[  \vbf \left( \frac{1}{2}\rh|\vbf|^2  + p \right) \right]\cdot\nabla_x\Psi\, \rd x\rd s \\
 & + \int_0^t \tht\int_{\dom} \sum_{k\geq1} \frac{1}{2\rho}|\ti\G(\rho,\rh\vbf) (e_k)|^2\, \Psi\, \rd x \rd s + \int_0^t \tht\int_{\dom}  \vbf \cdot \ti\G(\rho,\rh\vbf) \, \Psi\, \rd x \rd W_s = 0.
  \end{split}
 \end{align}

 \end{theorem}
%
%
%
%
 
 \begin{proof}[Proof of Theorm \ref{inhom-energyBLN-theorem}]
 
 Mollifying the system of  equations \eqref{cont-Eq}-\eqref{dv-vel}, we  obtain
  \begin{align} 
  & \rd\rho^{\e} + \dv_x(\rh\vbf)^{\e}\, \rd t = 0,\label{cont-Eq-Mof}\\
  & \rd(\rh\vbf)^{\e} + \dv_x\left(  \rh\vbf\otimes\vbf \right)^{\e}\rd t + \nabla_x p^{\e}\, \rd t = \ti\G(\rho,\rh\vbf)^{\e} \rd W_t,\label{moment-Eq-Mof}\\
  & \dv_x  \vbf^{\e} = 0\label{dv-vel-Mof}.
 \end{align}
  We observe that $\rho^\e\geq \bar{r}$ for $\e>0$. Let $\mbf:=\rh\vbf$. Note that, now onwards in the proof we work with $(\rho,\mbf)$ variable instead of $(\rho,\vbf)$, because \eqref{moment-Eq-Mof} is driven by Wiener process, hence we avoid time mollification and work with only space mollified version of the system. Now, from \eqref{cont-Eq-Mof}
 \[\rd\rho^{\e} = - \dv_x\mbf^{\e}\, \rd t. \]
 Applying It\^o's formula, for the function $r\mapsto\frac{1}{2r}$,
 \begin{equation}\label{1ito-ro}
 \rd\left( \frac{1}{2\rho^{\e}}\right) = \frac{1}{2(\rho^{\e})^2} \dv_x\mbf^{\e}\, \rd t. 
 \end{equation}
 Again, from \eqref{moment-Eq-Mof}
 \begin{align*}
  &  \rd\mbf^{\e}\\
  &= - \dv_x\left( \frac{\mbf\otimes\mbf}{\rho}\right)^{\e}\rd t - \nabla_x p^{\e}\, \rd t + \ti\G(\rho,\mbf)^{\e} \rd W_t\\
  & =  \left[-\dv_x\left( \frac{\mbf^{\e}\otimes\mbf^{\e}}{\rho^{\e}}\right) - \nabla_x p^{\e} + \mathcal R^{\e}\right]\, \rd t + \ti\G(\rho,\mbf)^{\e} \rd W_t,
 \end{align*}
 where $\mathcal{R}^\e$ is defined as,
 \begin{equation*}
 \mathcal R^{\e} := \dv_x\left( \frac{\mbf^{\e}\otimes\mbf^{\e}}{\rho^{\e}}\right) - \dv_x\left( \frac{\mbf\otimes\mbf}{\rho}\right)^{\e}.  
 \end{equation*}
 Applying It\^o's formula \cite{ito}, for the function $r\mapsto |r|^2$,
 \begin{align}\label{2ito-m-epsl-sqr}
  \begin{split}
  & \rd \left( |\mbf^{\e}|^2\right)\\
  & = \left( 2 \mbf^{\e} \cdot\left[-\dv_x\left( \frac{\mbf^{\e}\otimes\mbf^{\e}}{\rho^{\e}}\right) - \nabla_x p^{\e} + \mathcal R^{\e}\right] +\sum_{k\geq1} |\ti\G(\rho,\mbf)^{\e}(e_k)|^2 \right) \rd t\\
  & \quad + 2 \mbf^{\e}\cdot \ti\G(\rho,\mbf)^{\e} \rd W_t. 
 \end{split}
 \end{align}
 By It\^o's product rule \cite[Proposition 2.4.2., chapter 2]{bfh-book}, between \eqref{1ito-ro} and \eqref{2ito-m-epsl-sqr}
 \begin{align}\label{ito-prod-m-ro}
  \begin{split}
  & \rd\left( \frac{|\mbf^{\e}|^2}{2\rho^{\e}}\right)\\
  & = \left( \frac{\mbf^{\e}}{\rho^{\e}} \cdot \left[-\dv_x\left( \frac{\mbf^{\e}\otimes\mbf^{\e}}{\rho^{\e}}\right) - \nabla_x p^{\e} + \mathcal R^{\e}\right] +\sum_{k\geq1} \frac{1}{2\rho^{\e}}|\ti\G(\rho,\mbf)^{\e}(e_k)|^2 \right) \rd t\\
  & \quad\quad + \frac{\mbf^{\e}}{\rho^{\e}}\cdot \ti\G(\rho,\mbf)^{\e} \rd W_t + \frac{|\mbf^{\e}|^2}{2(\rho^{\e})^2} \dv_x\mbf^{\e}\, \rd t\\
  & = -(\dv_x\mbf^{\e})\frac{|\mbf^{\e}|^2}{2(\rho^{\e})^2}\, \rd t - \mbf^{\e}\cdot \nabla_x\left( \frac{\mbf^{\e}}{\rho^{\e}}\right) \cdot \frac{\mbf^{\e}}{\rho^{\e}}\, \rd t - \frac{\mbf^{\e}}{\rho^{\e}}\cdot \nabla_x p^{\e}\, \rd t + \mathcal R^{\e}\cdot \frac{\mbf^{\e}}{\rho^{\e}}\, \rd t\\
  & \quad\quad + \sum_{k\geq1} \frac{1}{2\rho^{\e}}|\ti\G(\rho,\mbf)^{\e}(e_k)|^2\, \rd t + \frac{\mbf^{\e}}{\rho^{\e}}\cdot \ti\G(\rho,\mbf)^{\e} \rd W_t\\
  & = -(\dv_x\mbf^{\e})\frac{|\mbf^{\e}|^2}{2(\rho^{\e})^2}\, \rd t - \mbf^{\e}\cdot \nabla_x \left( \frac{|\mbf^{\e}|^2}{2(\rho^{\e})^2}\right) \rd t -   \frac{\mbf^{\e}}{\rho^{\e}}\cdot \nabla_x p^{\e}\, \rd t\\ 
  & \quad\quad + \mathcal R^{\e}\cdot \frac{\mbf^{\e}}{\rho^{\e}}\, \rd t + \sum_{k\geq1} \frac{1}{2\rho^{\e}}|\ti\G(\rho,\mbf)^{\e}(e_k)|^2\, \rd t + \frac{\mbf^{\e}}{\rho^{\e}}\cdot \ti\G(\rho,\mbf)^{\e} \rd W_t\\
  & = -\dv_x \left[ \mbf^{\e}\left( \frac{|\mbf^{\e}|^2}{2(\rho^{\e})^2}\right)\right]\,\rd t - \frac{\mbf^{\e}}{\rho^{\e}}\cdot \nabla_x p^{\e}\, \rd t + \mathcal R^{\e}\cdot \frac{\mbf^{\e}}{\rho^{\e}}\, \rd t\\
  & \quad\quad + \sum_{k\geq1} \frac{1}{2\rho^{\e}}|\ti\G(\rho,\mbf)^{\e}(e_k)|^2\, \rd t + \frac{\mbf^{\e}}{\rho^{\e}}\cdot \ti\G(\rho,\mbf)^{\e} \rd W_t.
  \end{split}
 \end{align}
 Fix a $\theta\in C_c^\f(0,\f)$ and a $\Psi\in C_c^\f(\dom)$. Then integrating both sides of \eqref{ito-prod-m-ro} over  time and space, against $\theta\Psi$ and applying stochastic Fubini \cite[Theorem 4.33]{dpZ}, we obtain
 \begin{align}\label{sps-tm-int}
  \begin{split}
  & - \int_0^t \pa_t\tht \int_{\dom} \left( \frac{|\mbf^{\e}|^2}{2\rho^{\e}}\right)\Psi\, \rd x\rd s\\
  & = - \int_0^t \tht\int_{\dom} \dv_x \left[ \mbf^{\e}\left( \frac{|\mbf^{\e}|^2}{2(\rho^{\e})^2}\right)\right]\Psi\, \rd x\rd s\\
  & \quad\quad - \int_0^t \tht\int_{\dom} \frac{\mbf^{\e}}{\rho^{\e}} \cdot \nabla_x p^{\e}\, \Psi\ \rd x \rd s + \int_0^t \tht\int_{\dom} \mathcal R^{\e}\cdot \frac{\mbf^{\e}}{\rho^{\e}}\, \Psi\, \rd x\rd s\\
  & \quad\quad + \int_0^t \tht\int_{\dom} \sum_{k\geq1} \frac{1}{2\rho^{\e}}|\ti\G(\rho,\mbf)^{\e}(e_k)|^2\, \Psi\, \rd x \rd s + \int_0^t \tht\int_{\dom} \frac{\mbf^{\e}}{\rho^{\e}}\cdot \ti\G(\rho,\mbf)^{\e}\, \Psi\, \rd x \rd W_s.
  \end{split}
 \end{align}
 Note that,
 \begin{equation}\label{dv-clc2ndTrm}
 \dv_x\left( \frac{\mbf^{\e}}{\rho^{\e}} p^{\e}\right) = \frac{\mbf^{\e}}{\rho^{\e}} \cdot \nabla_x p^{\e} + \dv_x \left(\frac{\mbf^{\e}}{\rho^{\e}}\right)\, p^{\e}.
 \end{equation}
 Therefore from \eqref{dv-clc2ndTrm}, by applying integration by parts, 2nd term of r.h.s. of \eqref{sps-tm-int} becomes
 \begin{align}\label{dv-m-rho-p}
  \begin{split}
  & \int_0^t\tht\int_{\dom} \frac{\mbf^{\e}}{\rho^{\e}} \cdot \nabla_x p^{\e}\, \Psi\, \rd x\rd s\\ 
  & = \int_0^t\tht\int_{\dom} \dv_x\left( \frac{\mbf^{\e}}{\rho^{\e}} p^{\e}\right)  \Psi\, \rd x\rd s - \int_0^t\tht\int_{\dom} \dv_x \left(\frac{\mbf^{\e}}{\rho^{\e}}\right)\, p^{\e} \, \Psi\, \rd x\rd s\\
  & = - \int_0^t\tht\int_{\dom}  \left( \frac{\mbf^{\e}}{\rho^{\e}} p^{\e}\right) \cdot \nabla_x\Psi\, \rd x\rd s - \int_0^t\tht\int_{\dom}   \left[\dv_x\left(\frac{\mbf^{\e}}{\rho^{\e}}\right) - \dv_x\left( \frac{\mbf}{\rho}\right)^{\e}\right]\, p^{\e} \, \Psi\, \rd x\rd s.
  \end{split}
 \end{align}
 Hence, plugging \eqref{dv-m-rho-p} in \eqref{sps-tm-int}, we get
 \begin{align}\label{sps-tm-int-arnge}
  \begin{split}
  & - \int_0^t \pa_t\tht \int_{\dom} \left( \frac{|\mbf^{\e}|^2}{2\rho^{\e}}\right)\Psi\, \rd x\rd s\\
  & =  \int_0^t \tht\int_{\dom}  \left[ \mbf^{\e}\left( \frac{|\mbf^{\e}|^2}{2(\rho^{\e})^2}\right)\right]\cdot\nabla_x\Psi\, \rd x\rd s + \int_0^t\tht\int_{\dom}  \left( \frac{\mbf^{\e}}{\rho^{\e}}  p^{\e}\right) \cdot \nabla_x\Psi\, \rd x\rd s\\
  & \quad\quad + \int_0^t\tht\int_{\dom}  \left[\dv_x\left(\frac{\mbf^{\e}}{\rho^{\e}}\right) - \dv_x\left( \frac{\mbf}{\rho}\right)^{\e}\right]\, p^{\e} \, \Psi\, \rd x\rd s      + \int_0^t \tht\int_{\dom} \mathcal R^{\e}\cdot \frac{\mbf^{\e}}{\rho^{\e}}\, \Psi\, \rd x\rd s\\
  & \quad\quad + \int_0^t \tht\int_{\dom} \sum_{k\geq1} \frac{1}{2\rho^{\e}}|\ti\G(\rho,\mbf)^{\e}(e_k)|^2\, \Psi\, \rd x \rd s + \int_0^t \tht\int_{\dom} \frac{\mbf^{\e}}{\rho^{\e}}\cdot \ti\G(\rho,\mbf)^{\e}\, \Psi\, \rd x \rd W_s\\
  & =  \int_0^t \tht\int_{\dom}  \left[ \frac{\mbf^{\e}}{\rho^{\e}}\left( \frac{|\mbf^{\e}|^2}{2\rho^{\e}} + p^{\e} \right)\right]\cdot\nabla_x\Psi\, \rd x\rd s \\
  & \quad\quad + \int_0^t\tht\int_{\dom}  \left[\dv_x\left(\frac{\mbf^{\e}}{\rho^{\e}}\right) - \dv_x\left( \frac{\mbf}{\rho}\right)^{\e}\right]\, p^{\e} \, \Psi\, \rd x\rd s  + \int_0^t \tht\int_{\dom} \mathcal R^{\e}\cdot \frac{\mbf^{\e}}{\rho^{\e}}\, \Psi\, \rd x\rd s\\
  & \quad\quad + \int_0^t \tht\int_{\dom} \sum_{k\geq1} \frac{1}{2\rho^{\e}}|\ti\G(\rho,\mbf)^{\e}(e_k)|^2\, \Psi\, \rd x \rd s + \int_0^t \tht\int_{\dom} \frac{\mbf^{\e}}{\rho^{\e}}\cdot \ti\G(\rho,\mbf)^{\e}\, \Psi\, \rd x \rd W_s.
  \end{split}
 \end{align}
 Now, as $\e\to0$, by lemma  \ref{lemma:limit_pass_drift1}, l.h.s. of \eqref{sps-tm-int-arnge}:
 \begin{equation}\label{conv-lhs}
 - \int_0^t \pa_t\tht \int_{\dom} \left( \frac{|\mbf^{\e}|^2}{2\rho^{\e}}\right)\Psi\, \rd x\rd s \longrightarrow - \int_0^t \pa_t\tht \int_{\dom} \left( \frac{|\mbf|^2}{2\rho}\right)\Psi\, \rd x\rd s,\,\Pp-\mbox{a.s.} 
 \end{equation}
 Then, as $\e\to0$,  by lemma  \ref{lemma:limit_pass_drift1}, 1st term on the r.h.s. of  \eqref{sps-tm-int-arnge}:
 \begin{equation}\label{conv-rhs1}
 \int_0^t \tht\int_{\dom}  \left[ \frac{\mbf^{\e}}{\rho^{\e}}\left( \frac{|\mbf^{\e}|^2}{2\rho^{\e}} + p^{\e}\right)\right]\cdot\nabla_x\Psi\, \rd x\rd s \longrightarrow  \int_0^t \tht\int_{\dom}  \left[ \frac{\mbf}{\rho}\left( \frac{|\mbf|^2}{2\rho} + p \right) \right]\cdot\nabla_x\Psi\, \rd x\rd s, \,\Pp-\mbox{a.s.} 
 \end{equation}
 Next, as $\e\to0$, by Commutator estimate lemma \ref{lemma:commutator}, 2nd term on the r.h.s. of  \eqref{sps-tm-int-arnge}:
 \begin{equation}\label{conv-rhs2}
   \int_0^t\tht\int_{\dom}  \left[\dv_x\left(\frac{\mbf^{\e}}{\rho^{\e}}\right) - \dv_x\left( \frac{\mbf}{\rho}\right)^{\e}\right]\, p^{\e} \, \Psi\, \rd x\rd s    \longrightarrow 0,\,\Pp-\mbox{a.s.} 
 \end{equation}
 again, as $\e\to0$, by Commutator estimate lemma \ref{lemma:commutator}, 3rd term on the r.h.s. of  \eqref{sps-tm-int-arnge}:
 \begin{equation}\label{conv-rhs3}
  \int_0^t \tht\int_{\dom} \mathcal R^{\e}\cdot \frac{\mbf^{\e}}{\rho^{\e}}\, \Psi\, \rd x\rd s \longrightarrow 0,\,\Pp-\mbox{a.s.} 
 \end{equation}
 Then, as $\e\to0$, by lemma \ref{lemma:limit_pass_drift1}, 4th term on the r.h.s. of  \eqref{sps-tm-int-arnge}:
 \begin{equation}\label{conv-rhs4}
 \int_0^t \tht\int_{\dom} \sum_{k\geq1} \frac{1}{2\rho^{\e}}|\ti\G(\rho,\mbf)^{\e}(e_k)|^2\, \Psi\, \rd x \rd s \longrightarrow \int_0^t \tht\int_{\dom} \sum_{k\geq1} \frac{1}{2\rho}|\ti\G(\rho,\mbf) (e_k)|^2\, \Psi\, \rd x \rd s,  
 \end{equation}
 $\Pp$-a.s. At the end, as $\e\to0$, by lemma \ref{lemma:limit_pass_diff}, 5th term on the r.h.s. of  \eqref{sps-tm-int-arnge}:
 \begin{equation}\label{conv-rhs5}
 \int_0^t \tht\int_{\dom} \frac{\mbf^{\e}}{\rho^{\e}}\cdot \ti\G(\rho,\mbf)^{\e}\, \Psi\, \rd x \rd W_s \longrightarrow \int_0^t \tht\int_{\dom} \frac{\mbf}{\rho}\cdot \ti\G(\rho,\mbf) \, \Psi\, \rd x \rd W_s, 
 \end{equation}
 converges in mean square, hence in probability. Therefore, as $\e\to0$, combining \eqref{conv-lhs} - \eqref{conv-rhs5} in \eqref{sps-tm-int-arnge} and replacing $\mbf$ with $\rh\vbf$, for a subsequence, we obtain  energy-balance equation \eqref{enr-inhom-bln-EQUATION}, in $\Pp$- a.s. sense, for each $t\in[0,\tau)$. This concludes the proof.

 \end{proof}

	\section*{Acknowledgements} 
	
	Authors thank the Department of Atomic Energy, Government of India, under project no. 12-R\&D-TFR-5.01-0520 for the support. SSG would also like to acknowledge the Inspire faculty-research grant DST/INSPIRE/04/2016/000237.

	\end{document}